%% file: CCFPV2_2.tex
\documentclass[12pt]{article}
\pdfoutput=1

\usepackage{amsmath}
\usepackage[utf8]{inputenc}
\usepackage[UKenglish]{babel}

\usepackage[normalem]{ulem} 
\usepackage{aliascnt}
\usepackage{amssymb}
\usepackage{amsthm}
\usepackage{faktor}
\usepackage{mathtools}
\usepackage{amsfonts}
\usepackage{enumerate}
\usepackage{enumitem}
\usepackage{multicol}
\usepackage{float}
\usepackage{tikz}
\usepackage{tikz-cd} 
\usepackage{comment}
\usepackage[round,comma,sort,longnamesfirst
]{natbib}
\usepackage{hyperref}
\hypersetup{
	colorlinks=true,       
	linkcolor=blue,          
	citecolor=olive,       
	filecolor=magenta,      
	urlcolor=blue          
}
\usepackage[capitalise, nameinlink]{cleveref}
\crefname{subsection}{subsection}{subsections}
\newcommand{\Z}{\mathbb{Z}}

\begin{document}

\def\N{\mathbb N} \def\Z{\mathbb Z} \def\Min{{\rm Min}}
\def\pgcd{{\rm pgcd}}


\title{{\bf Classification of Artin groups admitting retractions onto their parabolic subgroups}}

\author{
\textsc{B.A. Cisneros de la Cruz}\thanks{SECIHTI - UNAM,  Unidad Oaxaca del Instituto de matemáticas de la UNAM León No. 2, Oaxaca de Juárez, Oaxaca, México. 68000. Email: bruno@im.unam.mx},
\textsc{M. Cumplido}\thanks{Instituto de Matemáticas de la Universidad de Sevilla (IMUS) and Departmento de Álgebra, Facultad de Matemáticas, Universidad de Sevilla. Email: cumplido@us.es},
\textsc{I. Foniqi}\thanks{ TU Berlin, Institut für Mathematik, Straße des 17. Juni 136, 10623 Berlin, Germany. Email: foniqi@math.tu-berlin.de},
\textsc{L. Paris}\thanks{Universit\'e Bourgogne Europe, CNRS, IMB, UMR 5584, 21000 Dijon, France. Email: lparis@u-bourgogne.fr}
}


\date{}

\maketitle

\theoremstyle{plain}
\newtheorem{theorem}{Theorem}
\numberwithin{theorem}{section}

\newaliascnt{lemma}{theorem}
\newtheorem{lemma}[lemma]{Lemma}
\aliascntresetthe{lemma}
\providecommand*{\lemmaautorefname}{Lemma}

\newaliascnt{proposition}{theorem}
\newtheorem{proposition}[proposition]{Proposition}
\aliascntresetthe{proposition}
\providecommand*{\propositionautorefname}{Proposition}

\newaliascnt{corollary}{theorem}
\newtheorem{corollary}[corollary]{Corollary}
\aliascntresetthe{corollary}
\providecommand*{\corollaryautorefname}{Corollary}

\newaliascnt{conjecture}{theorem}
\newtheorem{conjecture}[conjecture]{Conjecture}
\aliascntresetthe{conjecture}
\providecommand*{\conjectureautorefname}{Conjecture}

\theoremstyle{remark}

\newaliascnt{claim}{theorem}
\newaliascnt{remark}{theorem}
\newtheorem{claim}[claim]{Claim}
\newtheorem{remark}[remark]{Remark}
\newaliascnt{notation}{theorem}
\newtheorem{notation}[notation]{Notation}
\aliascntresetthe{notation}
\providecommand*{\notationautorefname}{Notation}

\aliascntresetthe{claim}
\providecommand*{\claimautorefname}{Claim}

\aliascntresetthe{remark}
\providecommand*{\remarkautorefname}{Remark}

\newtheorem*{claim*}{Claim}
\theoremstyle{definition}

\newaliascnt{definition}{theorem}
\newtheorem{definition}[definition]{Definition}
\aliascntresetthe{definition}
\providecommand*{\definitionautorefname}{Definition}

\newaliascnt{example}{theorem}
\newtheorem{example}[example]{Example}
\aliascntresetthe{example}
\providecommand*{\exampleautorefname}{Example}

\newaliascnt{question}{theorem}
\newtheorem{question}[question]{Question}
\aliascntresetthe{question}
\providecommand*{\exampleautorefname}{Question}

\def\autorefspace{\hspace*{-0.5pt}}
\def\sectionautorefname{Section\autorefspace}
\def\subsectionautorefname{Section\autorefspace}
\def\subsubsectionautorefname{Section\autorefspace}
\def\figureautorefname{Figure\autorefspace}
\def\subfigureautorefname{Figure\autorefspace}
\def\tableautorefname{Table\autorefspace}
\def\equationautorefname{Equation\autorefspace}
\def\Itemautorefname{item\autorefspace}
\def\Hfootnoteautorefname{footnote\autorefspace}
\def\AMSautorefname{Equation\autorefspace}

\begin{abstract}
We classify the Artin groups that admit retractions onto all of their parabolic subgroups. Our approach relies on a detailed analysis of triangular subgroups, with a key ingredient being the classification of homomorphisms between dihedral Artin groups that map one of the standard generators to a standard generator. As a consequence, we show that whenever an Artin group admits retractions to parabolic subgroups, it also admits ordinary ones - that is, retractions that send each standard generator either to a standard generator or to the identity.
\end{abstract}

\noindent
{\small \textbf{AMS Subject Classification (2020):} 20F36.

}

\section{Introduction}

Retractions provide a fundamental mechanism for understanding how large algebraic structures relate to their natural substructures. In group theory, a retraction is a homomorphism from a group onto a subgroup that restricts to the identity on that 
subgroup; more precisely, if $G$ is a group and $H \leq G$ a subgroup, then a \emph{retraction} of $G$ onto $H$ is a group homomorphism 
\[
\varphi : G \longrightarrow H
\]
such that $\varphi|_{H} = \mathrm{id}_{H}$.
The existence of such maps allows the global structure of a group to be studied through its smaller components. In this article we study retractions of Artin groups onto some special subgroups, called parabolic subgroups. 

\medskip

Let~$S$ be a finite set.
A \emph{Coxeter matrix} over~$S$ is a square matrix $M = (m_{s,t})_{s,t \in S}$ indexed by the elements of~$S$, having coefficients in~$\N \cup \{ \infty \}$, and satisfying $m_{s,s} = 1$ for every $s \in S$,
and $m_{s,t} = m_{t,s} \ge 2$ for every $s,t \in S$, $s \neq t$. We denote by $\Pi (s,t, m)$ the alternating word $sts \cdots$ of length~$m$. The \emph{Artin group} associated to~$M$ is the group~$A =A_S$ defined by the following presentation.
\[
A_S = \langle S \mid \Pi (s,t, m_{s,t}) = \Pi (t,s, m_{s,t}), \text{ for } s,t \in S,\ s \neq t,\ m_{s,t} \neq \infty \rangle\,.
\]
These groups were introduced by Jacques Tits in \citep{TitsArtin}. 

The \emph{standard parabolic subgroup} (also called the \emph{special subgroup}) $A_X$  of $A_S$ is the subgroup generated by $X \subseteq S$. It coincides with the Artin group associated to the submatrix of $M$ obtained by restricting to the rows and columns indexed by $X$ \citep{van1983homotopy}. Any conjugate of a standard parabolic subgroup is called a \emph{parabolic subgroup}. These subgroups retain much of the combinatorial and geometric information of the ambient group, and their role in the structure theory of Artin groups is crucial. For instance, they play an important role in the construction of complexes used to study Artin groups such as the Salvetti complex \citep{SalvettiComplex,Paris2014}, the Deligne complex \citep{CharneyDavis}, or the complex of irreducible parabolic subgroups \citep{cumplido2019parabolic}.  

\medskip
We say that $A_S$ \emph{admits a retraction} onto $A_X$ if there exists a homomorphism 
\[
\varphi : A_S \to A_X \quad \text{such that} \quad \varphi|_{A_X} = \operatorname{id}_{A_X}.
\]
We say that $A_S$ is \emph{parabolic-retractable} if it admits a retraction onto every standard parabolic subgroup. 
If $A_S$ is parabolic-retractable, then $A_S$ admits retractions to all parabolic subgroups; indeed, let $\alpha A_X \alpha^{-1}$ be a parabolic subgroup for some $\alpha \in A_S$, and let $\rho_X \colon A_S \to A_X$ be a retraction. Then the map
\[
\rho_X^\alpha \colon A_S \to \alpha A_X \alpha^{-1}, 
\qquad  
\rho_X^\alpha(g) = \alpha \,\rho_X(\alpha^{-1} g \alpha)\,\alpha^{-1},
\]
is a well-defined retraction as well. Indeed, for $g \in A_S$, we have $\rho_X (\alpha^{-1} g \alpha) \in A_X$, hence
\[
\rho_X^\alpha (g) = \alpha \rho_X (\alpha^{-1} g \alpha) \alpha^{-1} \in \alpha A_X \alpha^{-1}\,.
\] 
Moreover, for $g \in \alpha A_X \alpha^{-1}$, we have $\alpha^{-1} g \alpha \in A_X$, hence
\[
\rho_X^\alpha (g) = \alpha \rho_X (\alpha^{-1} g \alpha) \alpha^{-1} = \alpha \alpha^{-1} g \alpha \alpha^{-1} = g\,.
\]

\medskip

The presence of retractions often leads to significant structural results. For example, they have been used to study intersections of parabolic subgroups and to establish various results in right-angled Artin groups \citep{antolin2015tits}, in even Artin groups \citep{AntolinFoniqi, antolin2023subgroups}, and in FC-type Artin groups \citep*{CCF}. In the even case, $A_S$ is always parabolic-retractable: for any $X \subseteq S$, there is an evident retraction~$\rho: A_S \rightarrow A_X$ defined by~$\rho(s)=1$ if~$s \notin X$ and~$\rho(s)=s$ otherwise \citep{antolin2015tits}. This does not apply in general. Beyond the even case, a classification of parabolic-retractable Artin groups was known only for FC-type.

\medskip

The aim of this work is to advance the systematic study of retractions in Artin groups. We classify the Artin groups that are parabolic-retractable and establish criteria expressed in terms of their defining Coxeter matrices. Our approach reduces the problem to the study of homomorphisms between dihedral Artin groups, which serves as the foundation for the analysis of triangular subgroups. This reduction allows us to extend and generalize previous results known for right-angled and even Artin groups, and to address the broader question of how retractions influence the global structure of Artin groups. As we will show, the classification of parabolic-retractable groups links algebraic properties of Artin groups to combinatorial conditions on their Coxeter matrices. We now state our main results.

\begin{theorem}\label{theoremA}
    Let $A_S$ be any Artin group. Then $A_S$ is parabolic-retractable if and only if, for any triple $a, b, c \in S$ of pairwise distinct elements, $A_{a,b,c}$ retracts onto $A_{b,c}$.
\end{theorem}

We say that $M$ is \emph{odd} if all its entries are different from $\infty$ and are odd numbers. We say that $M$ is \emph{retract-compatible} if it is odd and, for any triple $a, b, c \in S$ of pairwise distinct elements, we have, up to permutation, that $m_{a,b} = m_{a,c}$ and $m_{b,c}$ divides $m_{a,b}$. A Coxeter matrix of rank 1 is always considered to be retract-compatible.

We say that $M$ is \emph{parabolic-retract-compatible} if there exists a partition $S = T_1 \sqcup \dots \sqcup T_k$ such that 
\begin{enumerate}
    \item For all $i \in \{1, \dots, k\}$, $M_{T_i}$ is retract-compatible.
    \item For each pair $i, j \in \{1, \dots, k\}$ with $i \neq j$, there exists $n_{ij}$, which is either an even number or $\infty$, such that $m_{a,b} = n_{ij}$ for all $(a,b) \in T_i \times T_j$. 
\end{enumerate}

\begin{theorem}\label{theoremB}
  Let $M$ be a Coxeter matrix and let $A_S$ be its associated Artin group. Then $A_S$ is parabolic-retractable if and only if $M$ is parabolic-retract-compatible. 
\end{theorem}

We say that a retraction $\varphi : A_S \longrightarrow A_X$ is \emph{ordinary} if $\varphi(x) = x$ for every $x \in X$ and $\varphi(y) \in X \cup \{1\}$ for every $y \in S \setminus X$. That is, the homomorphism sends standard generators either to standard generators or to the identity. The next theorem answers an open question asked in \citep*[Remark~2.8]{CCF}.

\begin{theorem}\label{theoremC}
Let $A_S$ be a parabolic-retractable Artin group and let $X \subseteq S$.
There exists an 
ordinary retraction $\varphi : A_S \to A_X$.
\end{theorem}

The strategy we will follow is based on the classification of homomorphisms between dihedral Artin groups that send one standard generator to another standard generator, which we will describe in \autoref{section2}. Then, in \autoref{section3}, we will use this classification to study retractions on triangles and show that this is sufficient to analyze all possible retractions. Some examples are presented in what follows.

\paragraph{Acknowledgments.} 
The first author is supported by SECIHTI through the Basic and Frontier Science Project CF2023G106.
The second author was supported by the Spanish Ministry of Science and Innovation (MCIN/AEI/10.13039/501100011033/FEDER, EU) [grant number PID2022-138719NA-I00].
The third author acknowledges support from the EPSRC Fellowship grant EP/V032003/1 ``Algorithmic, topological and geometric aspects of infinite groups, monoids and inverse semigroups''.
The fourth author is partially supported by the French project ``CaGeT'' (ANR-25-CE40-4162) of the ANR.
The IMB receives support from the EIPHI Graduate School (contract ANR-17-EURE-0002).

\subsection*{Examples}\label{section:examples}

We now describe a family of parabolic-retractable Artin groups that can be defined recursively. 
We construct an Artin group on $n$ generators with a retract-compatible Coxeter matrix as follows:

\begin{enumerate}
    \item Choose an odd number $m_1$ and set $m_{1,j} = m_1$ for every $j > 1$.
    \item For each $i = 2, \dots, n$, choose $m_i$ such that $m_i$ divides $m_{i-1}$, and set 
    $m_{i,j} = m_i$ for every $j > i$.
\end{enumerate}

Since a Coxeter matrix is symmetric and all diagonal entries are equal to $1$, this is well defined. 
Also note that for every triple $\{i, j, k\}$ with $i < j < k$, we have $m_{i,j} = m_{i,k}$, 
which is divisible by $m_{j,k}$; hence the construction yields a retract-compatible Coxeter matrix.

We represent Artin groups graphically using Coxeter graphs: the vertex set is the generating set 
of the group, and we draw an edge labelled $m_{a,b}$, $a \neq b$, whenever $m_{a,b} \neq \infty$. 
An example of a group from the family described above can be found in 
\autoref{figure:example_retract-compatible}.

\begin{figure}[ht]
\centering
\def\svgwidth{0.5\textwidth}
{\scriptsize
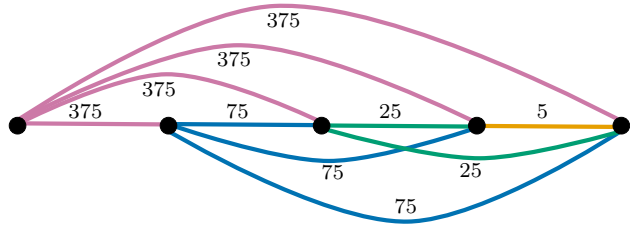}
\caption{An Artin group with a retract-compatible Coxeter matrix.}
\label{figure:example_retract-compatible}
\end{figure}

To construct a parabolic-retractable Artin group, take Coxeter graphs 
$\Gamma_1, \Gamma_2, \dots$, each arising from a retract-compatible Coxeter matrix. 
For every pair $(i,j)$ with $i \neq j$, connect all vertices of $\Gamma_i$ to all 
vertices of $\Gamma_j$ by edges all labeled with the same even integer, or do not put any edge between $\Gamma_i$ and $\Gamma_j$. An example of parabolic retractable Artin group is given in \autoref{figure:parabolic-retractable}.

\begin{figure}[ht]
\centering
\def\svgwidth{0.5\textwidth}
{\scriptsize
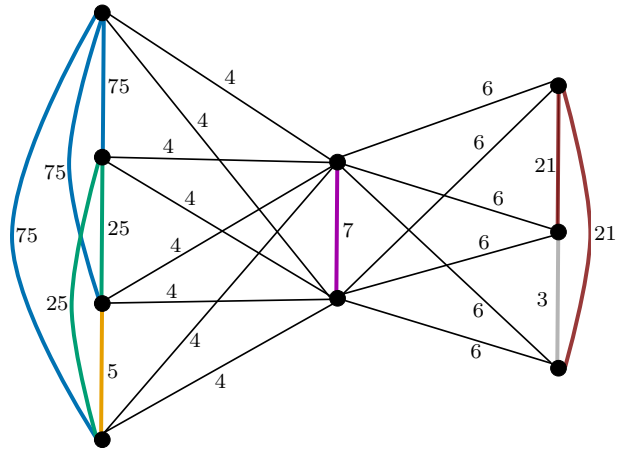}
\caption{A parabolic-retractable Artin group.}
\label{figure:parabolic-retractable}
\end{figure}

Note that this is not the only possible family of parabolic-retractable Artin groups, 
since the edge labels in a Coxeter graph do not need to be ordered by divisibility 
(see \autoref{figure:example_not_division}).

\begin{figure}[ht]
\centering
\def\svgwidth{0.2\textwidth}
{\scriptsize
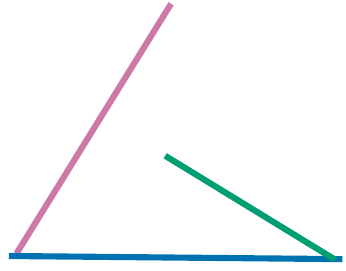}
\caption{Another Artin group with a retract-compatible Coxeter matrix whose labels are not ordered by division.}
\label{figure:example_not_division}
\end{figure}

\section{Classification of homomorphisms between dihedral Artin groups}\label{section2}


In our study, several elements of dihedral Artin groups play a special role, so we present them before stating our first main result. 
Let $A$ be a dihedral Artin group with the following presentation:
\[
A = \langle a_1, a_2 \mid \Pi (a_1, a_2, m_A) = \Pi (a_2, a_1, m_A) \rangle\,,
\]
where $m_A$ is an integer $\ge 3$ (so $m_A \neq \infty$ and $m_A\neq 2$). 
The \emph{Garside element} of $A$ is $\Delta_A = \Pi (a_1, a_2, m_A) = \Pi (a_2, a_1, m_A)$. 
It is known that the center $Z(A)$ of $A$ is a cyclic group generated by $\delta_A = \Delta_A^2$ if $m_A$ is odd and by $\delta_A = \Delta_A$ if $m_A$ is even (see \citealp{BrieskornSaito}). 
We also set $u_A = a_1 a_2$. 
Note that this definition implies that $a_1$ and $a_2$ are ordered, which will often be the case. 
Furthermore, this element will almost always appear up to conjugation, and $a_2 a_1 = a_2 u_A a_2^{-1}$ is conjugate to $u_A$.
Clearly, $u_A^{m_A} = \delta_A$ if $m_A$ is odd, and $u_A^{m_A/2} = \delta_A$ if $m_A$ is even.

\medskip

The following theorem classifies homomorphisms $\varphi : A \to B$ between dihedral Artin groups~$A$ and~$B$, under the condition that one generator of $A$ is mapped to a generator of $B$. Depending on the values of the defining parameters - labels, the possible images of the second generator vary significantly. The classification covers all combinations of odd, even, and infinite parameters, providing a complete description of how such homomorphisms behave.

\begin{theorem}\label{theorem:classification_of_homo}
    Let $A$ and $B$ be dihedral Artin groups with the following presentations:
\begin{align*}
A & = \langle a_1, a_2 \mid \Pi(a_1, a_2, m_A) = \Pi(a_2, a_1, m_A) \rangle, \\
B & = \langle b_1, b_2 \mid \Pi(b_1, b_2, m_B) = \Pi(b_2, b_1, m_B) \rangle,
\end{align*}
and let $\varphi : A \to B$ be a homomorphism that sends $a_1$ to $b_1$. Then the following classification holds.

\begin{enumerate}
\item If $m_A = \infty$, then $\varphi(a_2)$ can take any value in $B$.

\item If $m_A$ is odd and $m_B$ is even or $\infty$, then $\varphi(a_2) = b_1$.

\item If $m_A$ and $m_B$ are odd, then we have the following alternatives:
\begin{itemize}[leftmargin=*,itemsep=2pt,parsep=0pt,topsep=0pt,partopsep=0pt]
    \item $\varphi(a_2) = b_1$, or
    \item $m_B$ divides $m_A$ and there exists $t \in \mathbb{Z}$ such that $\varphi(a_2) = b_1^t b_2 b_1^{-t}$.
\end{itemize}

\item If $m_A$ is even and $m_B = \infty$, then there exists $t \in \mathbb{Z}$ such that $\varphi(a_2) = b_1^t$.

\item If $m_A$ is even and $m_B$ is odd, then we have the following alternatives:
\begin{itemize}[leftmargin=*,itemsep=2pt,parsep=0pt,topsep=0pt,partopsep=0pt]
    \item There exist $t_1, t_2 \in \mathbb{Z}$ such that $\varphi(a_2) = \delta_B^{t_1} b_1^{t_2}$,
    \item $m_A$ is divisible by $4$ and there exist $\beta \in B$ and $t \in \mathbb{Z}$ such that 
    $\varphi(a_2) = b_1^{-1} \beta \Delta_B^t \beta^{-1}$,
    \item There exist $\beta \in B$ and $t \in \mathbb{Z}$ such that $\varphi(a_2) = b_1^{-1} \beta u_B^{t} \beta^{-1}$ and, if $p = \gcd(t, m_B)$ and $q = \frac{m_B}{p}$, then $q$ divides $m_A/2$.
\end{itemize}

\item If $m_A$ and $m_B$ are even, and $m_B \neq 2$, then we have the following alternatives:
\begin{itemize}[leftmargin=*,itemsep=2pt,parsep=0pt,topsep=0pt,partopsep=0pt]
    \item There exist $t_1, t_2 \in \mathbb{Z}$ such that $\varphi(a_2) = \delta_B^{t_1} b_1^{t_2}$,
    \item There exist $\beta \in B$ and $t \in \mathbb{Z}$ such that $\varphi(a_2) = b_1^{-1} \beta u_B^{t} \beta^{-1}$ and, if $p = \gcd(t, m_B/2)$ and $q = \frac{m_B}{2p}$, then $q$ divides $m_A/2$.
\end{itemize}

\item If $m_A$ is even and $m_B = 2$, then $\varphi(a_2)$ can be any element of $B$.

\end{enumerate}

\end{theorem}

It is straightforward to verify that each of the above possibilities for $\varphi(a_2)$ defines a well-defined homomorphism $\varphi$.
We will prove this theorem in \Cref{subsec: Proof of dihedral_homomorphisms}, using some results on Bass-Serre theory from \Cref{subsec: Bass-Serre}, and a lemma about centralizers in dihedral Artin groups from \Cref{subsec: missing lemma}.

\subsection{Bass-Serre Tree}\label{subsec: Bass-Serre}

\noindent
Let $H$ and $K$ be two groups, and let $G = H \ast K$.
We define a graph $T$ as follows.
For each left coset $gH$ of $H$ with $g \in G$, there is a vertex $v(gH)$ in $T$, and for each left coset $gK$ of $K$ with $g \in G$, there is a vertex $v(gK)$ in $T$.
For each $g \in G$, there is an edge $e(g)$ in $T$ connecting $v(gH)$ to $v(gK)$.
It is known from \citep{Serre} that $T$ is a tree, called the \emph{Bass–Serre tree} (associated with the free product decomposition), and that we have an action of $G$ on $T$ defined by
\[
g \cdot v(hH) = v(ghH), \quad g \cdot v(hK) = v(ghK), \quad g \cdot e(h) = e(gh).
\]

\noindent
We endow $T$ with the usual metric, denoted by $d_T$, where all edges have length $1$.
Let $g \in G$.
The \emph{translation length} of $g$ is the number
\[
|g| = \inf \{ d_T(x, g(x)) \mid x \in V(T) \}.
\]
If $|g| = 0$, then we say that $g$ is \emph{elliptic}, and if $|g| \neq 0$, then we say that $g$ is \emph{hyperbolic}.
If $g \neq 1$ and $g$ is elliptic, then there exists a unique vertex $x \in V(T)$ such that $g(x) = x$ (this holds because the stabilizer of every edge in $T$ is trivial).
If $g$ is hyperbolic, then the set
\[
\Min(g) = \{ x \in V(T) \mid d_T(x, g(x)) = |g| \}
\]
generates a subgraph of $T$ isomorphic to $A_\infty$, called the \emph{axis} of $g$.
The graph $A_\infty$ is illustrated in the figure below.

\begin{figure}[H]
    \centering
\begin{tikzpicture}[scale=1.2, every node/.style={font=\small}]

  \draw (-4,0) -- (4,0);

  \foreach \i in {-3,-2,-1,1,2,3} {
    \draw (\i,0.1) -- (\i,-0.1); 
    \node[below] at (\i,-0.1) {$g^{\i} x$};
  }
  \draw (0,0.1) -- (0,-0.1);
  \node[below] at (0,-0.23) {$x$};
  \node at (-4.3,0) {$\dots$};
  \node at ( 4.3,0) {$\dots$};

  \draw[->,thick,blue] (-0.5,0.5) -- (0.5,0.5) node[midway,above] {$g$};

\end{tikzpicture}
\end{figure}

\noindent
In our proofs, we use the following results.


\begin{proposition}\label{prop: elliptic and hyperbolic elements}
    Let $H, K$ be two groups and let $G = H \ast K$.
\begin{itemize}[itemsep=2pt, parsep=2pt, topsep=2pt]
    \item[(1)] Let $g \in G \setminus \{1\}$.  
    If $g$ has finite order, then $g$ is elliptic.
    \item[(2)] Let $x_1, x_2 \in V(T)$ with $x_1 \neq x_2$, and let $h_1, h_2 \in G$.  
    Suppose $h_1, h_2$ are elliptic, $h_1(x_1) = x_1$, and $h_2(x_2) = x_2$.  
    Then $g = h_1h_2$ is a hyperbolic element of $G$, $|g| = 2 d_T(x_1, x_2)$, and $x_1, x_2$ lie on the axis of $g$.
\end{itemize}
\end{proposition}

Recall that if $a, b$ are two letters and $m \in \mathbb{N}_{\ge 2}$, we denote by $\Pi(a, b, m)$ the alternating product $aba \dots$ of length $m$.
In the next lemma we highlight how the parity of the defining integer affects the algebraic structure of a dihedral Artin group, setting the stage for later use of its center and quotient.

\begin{lemma}\label{lem: center and quotient in dihedral case}
Let $C = \langle c_1, c_2 \mid \Pi(c_1, c_2, m_C) = \Pi(c_2, c_1, m_C) \rangle$ be a dihedral Artin group.
\begin{itemize}[itemsep=2pt, parsep=2pt, topsep=2pt]
    \item[(1)] Suppose that $m_C$ is odd.   
    Then $C$ has the presentation $C = \langle \Delta_C, u_C \mid \Delta_C^2 = u_C^{m_C} \rangle$, and the quotient $C/Z(C)$ is the free product $C[2] \ast C[m_C]$, where $C[2]$ is a cyclic group of order $2$ generated by the class $\overline{\Delta_C}$, and $C[m_C]$ is a cyclic group of order $m_C$ generated by the class $\overline{u_C}$.
    
    \item[(2)] Suppose that $m_C$ is even and $m_C \neq 2$.  
    Then $C$ has the presentation $C = \langle c_1, u_C \mid c_1 u_C^{m_C/2} = u_C^{m_C/2} c_1 \rangle$, and the quotient $C/Z(C)$ is the free product $C[\infty] \ast C[m_C/2]$, where $C[\infty]$ is an infinite cyclic group generated by the class $\overline{c_1}$, and $C[m_C/2]$ is a cyclic group of order $m_C/2$ generated by the class $\overline{u_C}$.
\end{itemize}
\end{lemma}

\subsection{A Lemma}\label{subsec: missing lemma}

The following result will be used repeatedly in the proof of \autoref{theorem:classification_of_homo}. 
To the best of our knowledge, it is not explicitly stated in the literature, 
but it is a direct consequence of~\cite[Theorem~5.2]{Paris}.

\begin{lemma}\label{lem: missing lemma}
Let 
\[
A = \langle a_1, a_2 \mid \Pi(a_1, a_2, m_A) = \Pi(a_2, a_1, m_A) \rangle
\]
be a dihedral Artin group with $m_A < \infty$. 
Let $k_1, k_2 \in \mathbb{Z}$ with $k_2 \neq 0$. 
Then the centralizer of $\delta_A^{k_1} a_1^{k_2}$ in $A$ is the free abelian group generated by $\delta_A$ and $a_1$.
\end{lemma}

\begin{proof}
Let $Z$ be the centralizer of $\delta_A^{k_1} a_1^{k_2}$ in $A$. 
It is clear that $\delta_A$ and $a_1$ belong to $Z$, and it is also clear that 
$\delta_A$ and $a_1$ generate a free abelian group of rank $2$. 
Thus it suffices to prove that $Z \subseteq \langle \delta_A, a_1 \rangle$.

Since the center of $A$ is the cyclic group generated by $\delta_A$, 
it follows that $Z$ is also the centralizer of $a_1^{k_2}$ in $A$. Let $\beta \in Z$.

\medskip

\noindent
\textbf{Case 1: $m_A$ is even.}
Let $\mu_A = \Delta_A a_1^{-1} = \delta_A a_1^{-1}$. 
By~\cite[Theorem~5.2]{Paris}, there exist $t_1, t_2 \ge 0$ and $t_3 \in \mathbb{Z}$ such that
\[
\beta = \Delta_A^{-2t_1} \mu_A^{t_2} a_1^{t_3}.
\]
Hence
\[
\beta 
= \delta_A^{-2t_1 + t_2} a_1^{\,t_3 - t_2}
\in \langle \delta_A, a_1 \rangle.
\]

\medskip

\noindent
\textbf{Case 2: $m_A$ is odd.}
Let $\mu_{1,A} = \Delta_A a_1^{-1}$ and 
$\mu_{2,A} = \Delta_A a_2^{-1}$. 
Note that
\[
\mu_{2,A}\mu_{1,A}
= \Delta_A^2 a_1^{-2}
= \delta_A a_1^{-2}.
\]
By~\cite[Theorem~5.2]{Paris}, there exist $t_1, t_2 \ge 0$ and $t_3 \in \mathbb{Z}$ such that
\[
\beta 
= \Delta_A^{-2t_1} (\mu_{2,A}\mu_{1,A})^{t_2} a_1^{t_3}.
\]
Hence
\[
\beta 
= \delta_A^{-t_1 + t_2} a_1^{\,t_3 - 2t_2}
\in \langle \delta_A, a_1 \rangle.
\]

Therefore $Z \subseteq \langle \delta_A, a_1 \rangle$, 
and the proof is complete.
\end{proof}

\subsection{Case Analysis - Proof of \autoref{theorem:classification_of_homo}}\label{subsec: Proof of dihedral_homomorphisms}

\bigskip\noindent
We consider two Artin groups of dihedral type:
\begin{align*}
A & = \langle a_1, a_2 \mid \Pi(a_1, a_2, m_A) = \Pi(a_2, a_1, m_A) \rangle, \\
B & = \langle b_1, b_2 \mid \Pi(b_1, b_2, m_B) = \Pi(b_2, b_1, m_B) \rangle,
\end{align*}
and a homomorphism $\varphi : A \to B$ that sends $a_1$ to $b_1$.  
We will study the possibilities for such a homomorphism depending on the values of $m_A$ and $m_B$.

\bigskip\noindent
\textbf{Case 1:}  
\textit{$m_A = \infty$.  
Then $\varphi(a_2)$ can take any value in $B$.}

\bigskip\noindent
\textbf{Proof.}  
In this case, $A$ is a free group of rank $2$, freely generated by $\{a_1, a_2\}$.  
\qed

\bigskip\noindent
\textbf{Case 2.1:}  
\textit{$m_A$ is odd and $m_B = \infty$.  
Then $\varphi(a_2) = b_1$.}

\bigskip\noindent
\textbf{Proof.}  
Recall that the center of $A$ is an infinite cyclic group generated by $\delta_A = u_A^{m_A} = \Delta_A^2$.  
The element $\delta_A$ commutes with $a_1$, so $\varphi(\delta_A)$ commutes with $\varphi(a_1) = b_1$.  
Since $B$ is a free group freely generated by $\{b_1, b_2\}$, the centralizer of $b_1$ in $B$ is $\langle b_1 \rangle$, hence there exists $k \in \mathbb{Z}$ such that $\varphi(\delta_A) = b_1^k$.  
Suppose first that $k = 0$.  
We have
\[
\varphi(a_1a_2)^{m_A} = \varphi((a_1 a_2)^{m_A}) = \varphi(\delta_A) = 1.
\]
Since $B$ is torsion-free, it follows that $\varphi(a_1 a_2) = 1$, hence $\varphi(a_2) = \varphi(a_1)^{-1} = b_1^{-1}$.  
Applying $\varphi$ to the identity $\Pi(a_1, a_2, m_A) = \Pi(a_2, a_1, m_A)$ gives $\Pi(b_1, b_1^{-1}, m_A) = \Pi(b_1^{-1}, b_1, m_A)$.  
Since $m_A$ is odd, this implies $b_1 = b_1^{-1}$, a contradiction.  

Thus, $k \ne 0$.  
The centralizer of $\varphi(\delta_A) = b_1^k$ is $\langle b_1 \rangle$, and $\varphi(a_2)$ commutes with $\varphi(\delta_A)$, so there exists $\ell \in \mathbb{Z}$ such that $\varphi(a_2) = b_1^\ell$.  
Applying $\varphi$ to the identity $\Pi(a_1, a_2, m_A) = \Pi(a_2, a_1, m_A)$ gives $\Pi(b_1, b_1^\ell, m_A) = \Pi(b_1^\ell, b_1, m_A)$.  
Since $m_A$ is odd, this implies $b_1 = b_1^\ell$, so $\ell = 1$, as desired.  
\qed

\bigskip\noindent
{\bf Case 2.2:}
{\it Suppose that $m_A$ is odd, $m_B$ is even, and $m_B \neq 2$.
Then $\varphi (a_2) = b_1$.}

\bigskip\noindent
{\bf Proof.}
The element $\varphi (\delta_A)$ commutes with $\varphi (a_1) = b_1$, and the centralizer of $b_1$ in $B$ is generated by $\delta_B$ and $b_1$ (see \autoref{lem: missing lemma}). 
Hence there exist $k_1, k_2 \in \mathbb{Z}$ such that
\[
\varphi (\delta_A) = \delta_B^{k_1} b_1^{k_2}.
\]


Suppose $k_2 \neq 0$.  
The element $\varphi(a_2)$ commutes with $\varphi(\delta_A) = \delta_B^{k_1} b_1^{k_2}$, and the centralizer of $\delta_B^{k_1} b_1^{k_2}$ in $B$ is generated by $\delta_B$ and $b_1$ by \autoref{lem: missing lemma}, so there exist $\ell_1, \ell_2 \in \mathbb{Z}$ such that $\varphi(a_2) = \delta_B^{\ell_1} b_1^{\ell_2}$.  
Applying $\varphi$ to the identity $\Pi(a_1, a_2, m_A) = \Pi(a_2, a_1, m_A)$ gives $\Pi(b_1, \delta_B^{\ell_1} b_1^{\ell_2}, m_A) = \Pi(\delta_B^{\ell_1} b_1^{\ell_2}, b_1, m_A)$.  
Since $m_A$ is odd and $b_1$ commutes with $\delta_B^{\ell_1} b_1^{\ell_2}$, this implies $b_1 = \delta_B^{\ell_1} b_1^{\ell_2}$. This gives $\ell_1 = 0$ and $\ell_2 = 1$, so $\varphi(a_2) = b_1$.


Thus we may assume $k_2 = 0$.
In this situation,
\[
\varphi (\delta_A) = \delta_B^{k_1} \in Z(B),
\]
so $\varphi$ induces a homomorphism
\[
\bar \varphi : A / Z(A) \;\cong\; A[2] \ast A[m_A] \;\longrightarrow\; B / Z(B) \;\cong\; B[\infty] \ast B[m_B/2].
\]
We compute
\[
\bar \varphi (\overline{u_A})^{m_A} = \bar \varphi (\overline{u_A}^{\,m_A}) = 1,
\]
and by \autoref{prop: elliptic and hyperbolic elements} it follows that $\bar \varphi (\overline{u_A})$ belongs to a conjugate of $B[m_B/2]$ in $B/Z(B)$.
Hence there exist $\beta \in B$ and $t_1, t_2 \in \mathbb{Z}$ such that
\[
\varphi (u_A) = \beta u_B^{t_1} \beta^{-1} \delta_B^{t_2} 
              = \beta u_B^{t_1+t_2 m_B/2} \beta^{-1}.
\]

Let $z_B : B \to \mathbb{Z} \times \mathbb{Z}$ be the homomorphism defined by $z_B(b_1) = (1,0)$ and $z_B(b_2) = (0,1)$.
Since $\varphi(a_1) = b_1$, we have $z_B (\varphi (a_1)) = (1,0)$.
Moreover, because $a_2$ is conjugate to $a_1$ in $A$, $\varphi(a_2)$ is conjugate to $b_1$ in $B$, and therefore $z_B (\varphi(a_2)) = (1,0)$ as well.
Thus
\[
z_B (\varphi (u_A)) 
   = z_B (\varphi (a_1)) + z_B (\varphi (a_2)) 
   = (2,0).
\]

On the other hand,
\[
z_B \!\left(\beta u_B^{\,t_1+t_2 m_B/2} \beta^{-1}\right) 
   = (\,t_1+t_2 m_B/2,\, t_1+t_2 m_B/2\,),
\]
which is never equal to $(2,0)$.
This contradicts the equality 
$\varphi (u_A) = \beta u_B^{t_1+t_2 m_B/2} \beta^{-1}$.
Hence the assumption $k_2 = 0$ is impossible.
\qed

\bigskip\noindent
{\bf Case 2.3:}
{\it Suppose that $m_A$ is odd and $m_B = 2$.
Then $\varphi(a_2) = b_1$.}

\bigskip\noindent
{\bf Proof.}
Applying $\varphi$ to the defining relation
\[
\Pi(a_1, a_2, m_A) \;=\; \Pi(a_2, a_1, m_A)
\]
yields
\[
\Pi(b_1, \varphi(a_2), m_A) \;=\; \Pi(\varphi(a_2), b_1, m_A).
\]
Since $B$ is abelian and $m_A$ is odd, this equality forces $\varphi(a_2) = b_1$.
\qed

\bigskip\noindent
\textbf{Case 3:}  
\textit{$m_A$ and $m_B$ are both odd.  
One has the following alternatives:}
\begin{itemize}[itemsep=2pt, parsep=2pt, topsep=2pt]
    \item $\varphi(a_2) = b_1$,
    \item \emph{or $m_B$ divides $m_A$ and there exists $t \in \mathbb{Z}$ such that $\varphi(a_2) = b_1^t b_2 b_1^{-t}$.}
\end{itemize}

\bigskip\noindent
\textbf{Proof.}  
The element $\varphi(\delta_A)$ commutes with $\varphi(a_1) = b_1$, and the centralizer of $b_1$ in $B$ is generated by $b_1$ and $\delta_B$, so there exist $k_1, k_2 \in \mathbb{Z}$ such that $\varphi(\delta_A) = \delta_B^{k_1} b_1^{k_2}$.  
From here, the proof splits into two cases depending on whether $k_2 \neq 0$ or $k_2 = 0$. 

\bigskip\noindent
Suppose $k_2 \neq 0$. This case is analogous to the case $k_2 \neq 0$ in Case 2.2.

The element $\varphi(a_2)$ commutes with $\varphi(\delta_A) = \delta_B^{k_1} b_1^{k_2}$, and the centralizer of $\delta_B^{k_1} b_1^{k_2}$ in $B$ is generated by $\delta_B$ and $b_1$, so there exist $\ell_1, \ell_2 \in \mathbb{Z}$ such that $\varphi(a_2) = \delta_B^{\ell_1} b_1^{\ell_2}$.  
Applying $\varphi$ to the identity $\Pi(a_1, a_2, m_A) = \Pi(a_2, a_1, m_A)$ gives $\Pi(b_1, \delta_B^{\ell_1} b_1^{\ell_2}, m_A) = \Pi(\delta_B^{\ell_1} b_1^{\ell_2}, b_1, m_A)$.  
Since $m_A$ is odd and $b_1$ commutes with $\delta_B^{\ell_1} b_1^{\ell_2}$, this implies $b_1 = \delta_B^{\ell_1} b_1^{\ell_2}$. This gives $\ell_1 = 0$ and $\ell_2 = 1$, so $\varphi(a_2) = b_1$.

\bigskip\noindent
Suppose $k_2 = 0$.  
Since $\varphi(\delta_A) = \delta_B^{k_1} \in Z(B)$, $\varphi$ maps $Z(A)$ to $Z(B)$, so $\varphi$ induces a homomorphism
\[
\bar \varphi : A/Z(A) = A[2] \ast A[m_A] \;\longrightarrow\; B/Z(B) = B[2] \ast B[m_B].
\]
For $u_A = a_1a_2$, we have $\bar \varphi(\overline{u_A})^{m_A} = \bar \varphi(\overline{u_A}^{m_A}) = 1$, so by \autoref{prop: elliptic and hyperbolic elements}, $\bar \varphi(\overline{u_A})$ is elliptic.  
This means $\bar \varphi(\overline{u_A})$ belongs to a conjugate of $B[2]$ or of $B[m_B]$.  
But since $m_A$ is odd, $\bar \varphi(\overline{u_A})$ cannot belong to a conjugate of $B[2]$ unless it is trivial. If $\bar \varphi(\overline{u_A})$ is trivial, one has $\varphi(u_A) = \delta_B^{t}$, for some $t \in \Z$, which gives $\varphi(a_2) = \delta_B^{t}b_1^{-1}$. Again, applying $\varphi$ to the identity $\Pi (a_1, a_2, m_A) = \Pi (a_2, a_1, m_A)$ gives $\Pi (b_1, \delta_B^{t} b_1^{-1}, m_A) = \Pi (\delta_B^{t} b_1^{-1}, b_1, m_A)$. 
Since $m_A$ is odd and $b_1$ commutes with $\delta_B^{t} b_1^{-1}$, this implies $b_1 =  \delta_B^{t} b_1^{-1}$.
Such an equality would imply that the centralizer of $b_1^2 = \delta_B^t$ is the whole group $B$, contradicting \autoref{lem: missing lemma}.
So $\bar \varphi(\overline{u_A})$ belongs to a conjugate of $B[m_B]$.  
Thus, there exist $\beta_1 \in B$ and $t_1, t_2 \in \mathbb{Z}$ such that
\[
\varphi(u_A) = \beta_1 u_B^{t_1} \beta_1^{-1} \delta_B^{t_2} = \beta_1 u_B^{t_1 + m_B t_2} \beta_1^{-1}.
\]

Let $z_B : B \to \mathbb{Z}$ be the homomorphism sending $b_1$ and $b_2$ to $1$.  
We have $z_B(\varphi(a_1)) = z_B(b_1) = 1$.  
Since $a_2$ is conjugate to $a_1$ in $A$, $\varphi(a_2)$ is conjugate to $\varphi(a_1) = b_1$ in $B$, so $z_B(\varphi(a_2)) = 1$.  
It follows that
\[
z_B(\varphi(u_A)) = z_B(\varphi(a_1)) + z_B(\varphi(a_2)) = 2.
\]
On the other hand,
\[
z_B(\beta_1 u_B^{t_1 + m_B t_2} \beta_1^{-1}) = 2(t_1 + m_B t_2).
\]
Hence $t_1 + m_B t_2 = 1$, implying $\varphi(u_A) = \beta_1 u_B \beta_1^{-1}$.  
Since $\overline{u_B}$ has order $m_B$ and $\bar \varphi(\overline{u_A})^{m_A} = 1$, this implies that $m_B$ divides $m_A$.  

One has $\bar \varphi(\overline{\Delta_A})^{2} = \bar \varphi(\overline{\Delta_A}^{2}) = 1$, so by \autoref{prop: elliptic and hyperbolic elements}, $\bar \varphi(\overline{\Delta_A})$ is elliptic. This means $\bar \varphi(\overline{\Delta_A})$ belongs to a conjugate of $B[2]$ or of $B[m_B]$.  Since $m_B$ is odd, one has that $\bar \varphi(\overline{\Delta_A})$ belongs to a conjugate of $B[m_B]$ only if $\bar \varphi(\overline{\Delta_A})$ is trivial, in which case $\varphi(\Delta_A) = \delta_B^t = \Delta_B^{2t}$ for some integer $t$. On the other hand, if $\bar \varphi(\overline{\Delta_A})$ belongs to a conjugate of $B[2]$, 
there exists $\beta_2 \in B$ such that $\varphi(\Delta_A) = \beta_2 \Delta_B^{t_1+2t_2}\beta_2^{-1}$ for some integers $t_1, t_2$.

Combining these two cases, one has
$\varphi(\Delta_A) = \beta_2 \Delta_B^{t}\beta_2^{-1}$ for some $\beta_2 \in B$, and some $t\in \Z$. Using the homomorphism $z_B$ we obtain that $m_A=m_B\cdot t$ so we have that $\varphi(\Delta_A) = \beta_2 \Delta_B^{m_A/m_B}\beta_2^{-1}$. Since both $m_A$ and $m_B$ are odd, $\frac{m_A}{m_B}$ is odd and    
$\overline{\varphi}(\overline{\Delta_A}) = \overline{\beta_2} \,\overline{\Delta_B}\, \overline{\beta_2}^{-1}$.

Let $n_A$ and $n_B$ be the positive integers such that $m_A = 2n_A + 1$ and $m_B = 2n_B + 1$.  
Consider the action of $B/Z(B) = B[2] \ast B[m_B]$ on its Bass–Serre tree $T$.  
We have $\overline{b_1} = (\overline{u_B})^{-n_B} \overline{\Delta_B}$, so by \autoref{prop: elliptic and hyperbolic elements}, $\overline{b_1}$ is hyperbolic, $|\overline{b_1}| = 2$, and the set of vertices of $T$ lying on the axis of $\overline{b_1}$ is
\[
\{ v(\overline{b_1}^t B[2]) \mid t \in \mathbb{Z} \} \;\cup\; \{ v(\overline{b_1}^t B[m_B]) \mid t \in \mathbb{Z} \}.
\]
Also,
\[
\overline{b_1} = \bar \varphi(\overline{a_1}) = \bar \varphi(\overline{u_A})^{-n_A} \, \bar \varphi(\overline{\Delta_A}) 
= (\overline{\beta_1} \,\overline{u_B}^{-n_A}\, \overline{\beta_1}^{-1})(\overline{\beta_2}\,\overline{\Delta_B}\,\overline{\beta_2}^{-1}),
\]
so, by \autoref{prop: elliptic and hyperbolic elements}, $v(\overline{\beta_1} B[m_B])$ lies on the axis of $\overline{b_1}$.  
It follows that there exists $t \in \mathbb{Z}$ such that $\overline{\beta_1} B[m_B] = \overline{b_1}^t B[m_B]$.  
This implies $\beta_1$ is of the form $\beta_1 = b_1^t u_B^{s_1} \delta_B^{s_2}$ with $s_1, s_2 \in \mathbb{Z}$, hence
\[
\varphi(u_A) = \beta_1 u_B \beta_1^{-1} = b_1^t u_B b_1^{-t}.
\]
We conclude that
\[
\varphi(a_2) = \varphi(a_1)^{-1}\varphi(u_A) = b_1^t b_2 b_1^{-t}.
\]
\qed

\bigskip\noindent
{\bf Case 4.1:}
{\it $m_A$ is even, $m_A \neq 2$, and $m_B = \infty$.
Then there exists $t \in \Z$ such that $\varphi (a_2) = b_1^t$.}

\bigskip\noindent
{\bf Proof.}
The element $\varphi(\delta_A)$ commutes with $\varphi(a_1) = b_1$. 
Since $B$ is a free group freely generated by $\{b_1, b_2\}$, the centralizer of $b_1$ in $B$ is $\langle b_1 \rangle$. 
Hence, there exists $k \in \Z$ such that $\varphi(\delta_A) = b_1^k$.
If $k = 0$, then
\[
\varphi(u_A)^{m_A/2} = \varphi(u_A^{m_A/2}) = \varphi(\delta_A) = 1.
\]
Since $B$ is torsion-free, this implies $\varphi(u_A) = \varphi(a_1 a_2) = 1$, and thus
\[
\varphi(a_2) = \varphi(a_1)^{-1} = b_1^{-1}.
\]

If $k \neq 0$, then the centralizer of $\varphi(\delta_A) = b_1^k$ in $B$ is $\langle b_1 \rangle$, and $\varphi(a_2)$ commutes with $\varphi(\delta_A)$. 
Hence, there exists $t \in \Z$ such that $\varphi(a_2) = b_1^t$.
\qed

\bigskip\noindent
{\bf Case 4.2:}  
{\it $m_A = 2$ and $m_B = \infty$.  
Then there exists $t \in \Z$ such that $\varphi(a_2) = b_1^t$.}

\bigskip\noindent
{\bf Proof.}  
The centralizer of $b_1$ in $B$ is $\langle b_1 \rangle$, and since $\varphi(a_2)$ commutes with $\varphi(a_1) = b_1$, there exists $t \in \Z$ such that
$\varphi(a_2) = b_1^t$.
\qed

\bigskip\noindent
{\bf Case 5.1:}
{\it $m_A$ is even, $m_A \neq 2$, and $m_B$ is odd.
Then we have the following alternatives:
\begin{itemize}[itemsep=2pt,parsep=2pt,topsep=2pt]
\item
There exist $t_1, t_2 \in \Z$ such that $\varphi(a_2) = \delta_B^{t_1} b_1^{t_2}$,
\item
$m_A$ is divisible by $4$ and there exist $\beta \in B$ and $t \in \Z$ such that 
$\varphi(a_2) = b_1^{-1} \beta \Delta_B^t \beta^{-1}$,
\item
There exist $\beta \in B$ and $t \in \Z$ such that $\varphi(a_2) = b_1^{-1} \beta u_B^t \beta^{-1}$, and, if $p = \gcd(t, m_B)$ and $q = m_B/p$, then $q$ divides $m_A/2$.
\end{itemize}}

\bigskip\noindent
{\bf Proof.}
Note that these cases are not necessarily disjoint. We begin by verifying that each defines a homomorphism from $A$ to $B$.  

If there exist $t_1, t_2 \in \Z$ such that $\varphi(a_2) = \delta_B^{t_1} b_1^{t_2}$, then $\varphi(a_1)$ and $\varphi(a_2)$ commute, and since $m_A$ is even, one has: 
\[
\Pi(\varphi(a_1), \varphi(a_2), m_A) = \Pi(\varphi(a_2), \varphi(a_1), m_A).
\]  

Suppose $m_A$ is divisible by $4$ and there exist $\beta \in B$ and $t \in \Z$ such that $\varphi(a_2) = b_1^{-1} \beta \Delta_B^t \beta^{-1}$. Let $n_A = m_A/4$. Then
\begin{align*}
\Pi(\varphi(a_1), \varphi(a_2), m_A) 
&= (\varphi(a_1) \varphi(a_2))^{2 n_A} \\
&= (b_1 b_1^{-1} \beta \Delta_B^t \beta^{-1})^{2 n_A} \\
&= (\beta \Delta_B^t \beta^{-1})^{2 n_A} = \beta \Delta_B^{2 t n_A} \beta^{-1} \\
&= \beta \delta_B^{t n_A} \beta^{-1} = \delta_B^{t n_A} \in Z(B),
\end{align*}
showing it lies in the center of $B$.  
This clearly implies:
\begin{align*}
\Pi(\varphi(a_2), \varphi(a_1), m_A) = \delta_B^{t n_A}.
\end{align*}

Suppose there exist $\beta \in B$ and $t \in \Z$ such that $\varphi(a_2) = b_1^{-1} \beta u_B^t \beta^{-1}$ and, if $p = \gcd(t, m_B)$, then $q = m_B/p$ divides $m_A/2$. Let $q_1, q_2 \in \Z$ be such that $t = p q_1$ and $m_A/2 = q q_2$. Then
\begin{align*}
\Pi(\varphi(a_1), \varphi(a_2), m_A) 
&= (b_1 b_1^{-1} \beta u_B^t \beta^{-1})^{m_A/2} \\
&= \beta u_B^{t (m_A/2)} \beta^{-1} = \beta u_B^{p q_1 q q_2} \beta^{-1} \\
&= \beta u_B^{m_B q_1 q_2} \beta^{-1} = \beta \delta_B^{q_1 q_2} \beta^{-1} = \delta_B^{q_1 q_2} \in Z(B).
\end{align*}
As in the previous paragraph, this implies that $\Pi (\varphi (a_1), \varphi (a_2), m_A) = \Pi (\varphi (a_2), \varphi (a_1), m_A)$.

\bigskip\noindent
Now suppose $\varphi : A \to B$ is a homomorphism with $\varphi(a_1) = b_1$. The element $\varphi(\delta_A)$ commutes with $b_1$, and the centralizer of $b_1$ in $B$ is generated by $\delta_B$ and $b_1$, so there exist $k_1, k_2 \in \Z$ such that 
\[
\varphi(\delta_A) = \delta_B^{k_1} b_1^{k_2}.
\]

If $k_2 \neq 0$, then the centralizer of $\varphi(\delta_A)$ is again generated by $\delta_B$ and $b_1$, so there exist $t_1, t_2 \in \Z$ such that $\varphi(a_2) = \delta_B^{t_1} b_1^{t_2}$.  

If $k_2 = 0$, then $\varphi(\delta_A) = \delta_B^{k_1} \in Z(B)$, and $\varphi$ induces a homomorphism 
\[
\bar \varphi : A/Z(A) = A[\infty] \ast A[m_A/2] \to B/Z(B) = B[2] \ast B[m_B].
\]  
We have 
\[
\bar \varphi(\overline{u_A})^{m_A} = \bar \varphi(\overline{u_A}^{m_A}) = \bar \varphi(\overline{\delta_A}) = 1.
\]  
By \autoref{prop: elliptic and hyperbolic elements}, $\bar \varphi(\overline{u_A})$ lies in a conjugate of $B[2]$ or a conjugate of $B[m_B]$.  

If $\bar \varphi(\overline{u_A}) = 1$, then there exists $t_1 \in \Z$ such that $\varphi(u_A) = \delta_B^{t_1}$, hence 
\[
\varphi(a_2) = \varphi(a_1)^{-1} \varphi(u_A) = \delta_B^{t_1} b_1^{-1}.
\]  

If $\bar \varphi(\overline{u_A}) \neq 1$ and lies in a conjugate of $B[2]$, then there exists $\beta \in B$ such that $\bar \varphi(\overline{u_A}) = \bar \beta \overline{\Delta_B} \bar \beta^{-1}$. Since $\bar \beta \overline{\Delta_B}^{m_A/2} \bar \beta^{-1} = 1$ and $\overline{\Delta_B}$ has order $2$, we conclude that $m_A$ is divisible by $4$. There exists $t_1 \in \Z$ such that 
\[
\varphi(u_A) = \beta \Delta_B \beta^{-1} \delta_B^{t_1} = \beta \Delta_B^{1 + 2 t_1} \beta^{-1},
\]  
and thus
\[
\varphi(a_2) = b_1^{-1} \beta \Delta_B^{1 + 2 t_1} \beta^{-1}.
\]

If $\bar \varphi(\overline{u_A}) \neq 1$ and lies in a conjugate of $B[m_B]$, then there exist $\beta \in B$ and $t_1 \in \Z$ such that 
\[
\bar \varphi(\overline{u_A}) = \bar \beta \overline{u_B}^{t_1} \bar \beta^{-1}, \quad 1 \le t_1 \le m_B - 1.
\]  
Let $p = \gcd(t_1, m_B)$ and $q = m_B/p$. Since the order of $\bar \varphi(\overline{u_A})$ divides $m_A/2$, $q$ divides $m_A/2$. There exists $t_2 \in \Z$ such that 
\[
\varphi(u_A) = \beta u_B^{t_1} \beta^{-1} \delta_B^{t_2} = \beta u_B^{t_1 + t_2 m_B} \beta^{-1},
\]  
and hence
\[
\varphi(a_2) = b_1^{-1} \beta u_B^{t_1 + t_2 m_B} \beta^{-1},
\]  
noting that $\gcd(t_1 + t_2 m_B, m_B) = p = \gcd(t_1, m_B)$.
\qed

\bigskip\noindent
{\bf Case 5.2:}  
{\it $m_A = 2$ and $m_B$ is odd.
Then there exist $t_1, t_2 \in \Z$ such that $\varphi(a_2) = \delta_B^{t_1} b_1^{t_2}$.}

\bigskip\noindent
{\bf Proof.}  
The centralizer of $b_1$ in $B$ is $\langle \delta_B, b_1 \rangle$. Since $\varphi(a_2)$ commutes with $\varphi(a_1) = b_1$, there exist $t_1, t_2 \in \Z$ such that
\[
\varphi(a_2) = \delta_B^{t_1} b_1^{t_2}.
\]  
\qed

\bigskip\noindent
{\bf Case 6.1:}  
{\it $m_A$ and $m_B$ are even, with $m_A \neq 2$ and $m_B \neq 2$. Then one of the following alternatives holds:
\begin{itemize}[itemsep=2pt,parsep=2pt,topsep=2pt]
\item
There exist integers $t_1, t_2 \in \Z$ such that $\varphi(a_2) = \delta_B^{t_1} b_1^{t_2}$.
\item
There exist $\beta \in B$ and $t_1 \in \Z$ such that $\varphi(a_2) = b_1^{-1} \beta u_B^{t_1} \beta^{-1}$, and, if we set $p = \gcd(t_1, m_B/2)$ and $q = \frac{m_B}{2p}$, then $q$ divides $m_A/2$.
\end{itemize}}

\bigskip\noindent
{\bf Proof.}  
As in Case 5.1, one can check that a homomorphism of either type above is well-defined.  
Consider a homomorphism $\varphi: A \to B$ with $\varphi(a_1) = b_1$, and let us show that it satisfies one of the two alternatives.  

Since $\varphi(\delta_A)$ commutes with $\varphi(a_1) = b_1$, and the centralizer of $b_1$ in $B$ is generated by $\delta_B$ and $b_1$, there exist integers $k_1, k_2 \in \Z$ such that 
\[
\varphi(\delta_A) = \delta_B^{k_1} b_1^{k_2}.
\]

\medskip
\noindent
If $k_2 \neq 0$, the centralizer of $\varphi(\delta_A)$ in $B$ is generated by $\delta_B$ and $b_1$. Since $\varphi(a_2)$ commutes with $\varphi(\delta_A)$, there exist integers $t_1, t_2 \in \Z$ such that 
\[
\varphi(a_2) = \delta_B^{t_1} b_1^{t_2}.
\]  
This gives the first type of homomorphism.

\medskip
\noindent
Now for $k_2 = 0$, one has $\varphi(\delta_A) = \delta_B^{k_1} \in Z(B)$, so $\varphi$ induces a homomorphism
\[
\bar \varphi : A/Z(A) = A[\infty] \ast A[m_A/2] \longrightarrow B/Z(B) = B[\infty] \ast B[m_B/2].
\]  
Since
\[
\bar \varphi(\overline{u_A})^{m_A/2} = \bar \varphi(\overline{u_A}^{m_A/2}) = \bar \varphi(\overline{\delta_A}) = 1,
\]  
\autoref{prop: elliptic and hyperbolic elements} implies that $\bar \varphi(\overline{u_A})$ lies in a conjugate of $B[m_B/2]$. Hence, there exist $\beta \in B$ and $t_1 \in \{0, 1, \ldots, m_B/2 - 1\}$ such that
\[
\bar \varphi(\overline{u_A}) = \bar \beta \overline{u_B}^{t_1} \bar \beta^{-1}.
\]
Suppose first that $t_1 = 0$. Then $\bar \varphi(\overline{u_A}) = 1$, hence there exists $t_2 \in \Z$ such that $\varphi(u_A) = \varphi(a_1)\varphi(a_2) = \delta_B^{t_2}$, therefore $\varphi(a_2) = \varphi(a_1)^{-1}\delta_B^{t_2} = \delta_B^{t_2} b_1^{-1}$.

Now suppose that $t_1 \neq 0$.
Let $p = \gcd(t_1, m_B/2)$ and $q = \frac{m_B}{2p}$. Then the order of $\bar \varphi(\overline{u_A})$ is $q$, and since $\bar \varphi(\overline{u_A})^{m_A/2} = 1$, it follows that $q$ divides $m_A/2$.

Moreover, there exists $t_2 \in \Z$ such that
\[
\varphi(u_A) = \beta u_B^{t_1} \beta^{-1} \delta_B^{t_2} = \beta u_B^{t_1 + \frac{t_2 m_B}{2}} \beta^{-1}.
\]  
Therefore,
\[
\varphi(a_2) = \varphi(a_1)^{-1} \varphi(u_A) = b_1^{-1} \beta u_B^{t_1 + \frac{t_2 m_B}{2}} \beta^{-1}.
\]  
Finally, note that
\[
\gcd\Big(t_1 + \frac{t_2 m_B}{2}, \frac{m_B}{2}\Big) = \gcd(t_1, m_B/2) = p.
\]  
\qed

\bigskip\noindent
{\bf Case 6.2:}  
{\it $m_A = 2$, $m_B$ is even, and $m_B \neq 2$.  
Then there exist $t_1, t_2 \in \Z$ such that $\varphi(a_2) = \delta_B^{t_1} b_1^{t_2}$.}

\bigskip\noindent
{\bf Proof.}  
Since $a_2$ commutes with $a_1$, then $\varphi(a_2)$ commutes with $\varphi(a_2) = b_1$. Hence, by \autoref{lem: missing lemma} there exist $t_1, t_2 \in \Z$ such that $\varphi(a_2) = \delta_B^{t_1} b_1^{t_2}$.
\qed

\bigskip\noindent
{\bf Case 7:}  
{\it $m_A$ is even (including $m_A = 2$) and $m_B = 2$.  
Then $\varphi(a_2)$ can be any element of $B$.}

\bigskip\noindent
{\bf Proof.}  
Let $\beta \in B$ be arbitrary.  
Since $B$ is abelian and $m_A$ is even, we always have
\[
\Pi(b_1, \beta, m_A) = \Pi(\beta, b_1, m_A),
\]  
so there always exists a homomorphism $\varphi: A \to B$ such that 
\[
\varphi(a_1) = b_1 \quad \text{and} \quad \varphi(a_2) = \beta.
\]  
\qed

\section{Characterizing Parabolic-Retractable Artin Groups}\label{section3}

The following lemma will be central in the proof of \autoref{theoremA}. 

\begin{lemma}\label{lemma_0_induction}
    Let $A_S$ be an Artin group. Then $A_S$ is parabolic-retractable if and only if, for every $X \subseteq S$ and every $x \in X$, the subgroup $A_X$ retracts onto $A_{X \setminus \{x\}}$. 
\end{lemma}

\begin{proof}
Suppose first that $A_S$ is parabolic-retractable. In particular, for any subset $X \subseteq S$ and any $x \in X$, there exists a retraction
$\varphi: A_S \to A_{X \setminus \{x\}}$.
Restricting $\varphi$ to $A_X$ yields the desired retraction 
$\varphi|_{A_X} : A_X \to A_{X \setminus \{x\}}$.

Conversely, assume that for every $X \subseteq S$ and every $x \in X$, the subgroup $A_X$ retracts onto $A_{X \setminus \{x\}}$. We prove by induction on $k = |S \setminus X|$ that $A_S$ retracts onto $A_X$ for all subsets $X \subseteq S$.

For $k = 0$, the statement is trivial since the identity map gives a retraction $A_S \to A_S$.  
Assume now that $A_S$ retracts onto $A_X$ whenever $|S \setminus X| = k$. Let $X' \subseteq S$ with $|S \setminus X'| = k+1$. Choose $x \in S \setminus X'$ and set $X := X' \cup \{x\}$. Then $|S \setminus X| = k$, so by the induction hypothesis there exists a retraction $A_S \to A_X$. By assumption, $A_X$ retracts onto $A_{X'}$. Composing these two retractions gives the required retraction $A_S \to A_{X'}$.
\end{proof}

\begin{lemma}\label{lemma_1_odd_odd_odd}
Assume that $A_S$ is a parabolic-retractable Artin group. Let $a,b,c \in S$ be pairwise distinct. If $m_{a,b}$ and $m_{a,c}$ are odd, then $m_{b,c}$ is also odd.
\end{lemma}

\begin{proof}
Suppose, for contradiction, that $m_{b,c}$ is even or infinite. Since $A_S$ is parabolic-retractable, there exists a retraction 
$\varphi \colon A_{\{a,b,c\}} \to A_{\{b,c\}}$ 
with $\varphi(b) = b$ and $\varphi(c) = c$.  

Applying \autoref{theorem:classification_of_homo}(2) to the restriction of $\varphi$ on $A_{\{a,b\}}$, we deduce that $\varphi(a) = b$. On the other hand, applying \autoref{theorem:classification_of_homo}(2) to the restriction on $A_{\{a,c\}}$ forces $\varphi(a) = c$, a contradiction. Hence $m_{b,c}$ must be odd.
\end{proof}

\begin{lemma}\label{lemma_2_odd_even_odd}
Assume that $A_S$ is a parabolic-retractable Artin group. Let $a,b,c \in S$ be pairwise distinct. If $m_{a,b}$ is odd and $m_{a,c}$ is even, then $m_{b,c}$ is even.
\end{lemma}

\begin{proof}
By \autoref{lemma_1_odd_odd_odd}, the label $m_{b,c}$ cannot be odd. Suppose instead that $m_{b,c} = \infty$. Since $A_S$ is parabolic-retractable, there exists a retraction 
$\varphi \colon A_{\{a,b,c\}} \to A_{\{b,c\}}$ 
with $\varphi(b) = b$ and $\varphi(c) = c$.  

Applying \autoref{theorem:classification_of_homo}(2) to the restriction of $\varphi$ on $A_{\{a,b\}}$ yields $\varphi(a) = b$. Meanwhile, applying \autoref{theorem:classification_of_homo}(4) to the restriction on $A_{\{a,c\}}$ forces $\varphi(a) = c^t$ for some $t \in \Z$. These conditions are incompatible, giving a contradiction. Thus, $m_{b,c}$ must be even.
\end{proof}

\begin{lemma}\label{lemma_3_mbc_divides_odd_labels}
Assume that $A_S$ is a parabolic-retractable Artin group. Let $a,b,c \in S$ be pairwise distinct with $m_{a,b}, m_{a,c}$, and $m_{b,c}$ all odd. Then $m_{b,c}$ divides either $m_{a,b}$ or $m_{a,c}$.
\end{lemma}

\begin{proof}
Since $A_S$ is parabolic-retractable, there exists a retraction 
$\varphi \colon A_{\{a,b,c\}} \to A_{\{b,c\}}$ 
with $\varphi(b) = b$ and $\varphi(c) = c$.  

By \autoref{theorem:classification_of_homo}(3), restricting $\varphi$ to $A_{\{a,b\}}$ gives two possibilities:
\begin{itemize}
    \item[(b.1)] $\varphi(a)=b$;
    \item[(b.2)] $m_{b,c}$ divides $m_{a,b}$, and there exists $t_1 \in \mathbb{Z}$ with $\varphi(a)=b^{t_1}cb^{-t_1}$.
\end{itemize}
Similarly, restricting to $A_{\{a,c\}}$ gives:
\begin{itemize}
    \item[(c.1)] $\varphi(a)=c$;
    \item[(c.2)] $m_{b,c}$ divides $m_{a,c}$, and there exists $t_2 \in \mathbb{Z}$ with $\varphi(a)=c^{t_2}bc^{-t_2}$.
\end{itemize}

The cases (b.1) and (c.1) are incompatible, so at least one of (b.2) or (c.2) must hold. Hence $m_{b,c}$ divides either $m_{a,b}$ or $m_{a,c}$.
\end{proof}

\begin{lemma}\label{lemma_4_mbc_divides_odd_labels_all_cases}
Assume that $A_S$ is a parabolic-retractable Artin group. Let $a,b,c\in S$ be pairwise distinct with $m_{a,b}, m_{a,c}$, and $m_{b,c}$ all odd. Then, up to permuting $a,b,c$, we have $m_{a,b}=m_{a,c}$ and $m_{b,c}\mid m_{a,b}$.
\end{lemma}

\begin{proof}
By \autoref{lemma_3_mbc_divides_odd_labels} applied to the three ordered pairs, each of the three odd numbers
\[
x:=m_{b,c},\qquad y:=m_{a,c},\qquad z:=m_{a,b}
\]
divides at least one of the other two. Without loss of generality, assume $x\le y\le z$.

From the divisibility condition for $z$ we have $z\mid x$ or $z\mid y$. Since $z$ is the largest of the three, either possibility forces $z\le x$ or $z\le y$, hence $z=x$ or $z=y$. In both cases we conclude $y=z$.

Now apply the divisibility condition for $x$: we have $x\mid y$ or $x\mid z$. As $y=z$, this yields $x\mid y$. Translating back, $m_{a,b}=m_{a,c}$ and $m_{b,c}\mid m_{a,b}$, as claimed (up to permutation of $a,b,c$).
\end{proof}

\begin{lemma}\label{lemma_5_even_even_odd}
Assume that $A_S$ is a parabolic-retractable Artin group. Let $a,b,c \in S$ be pairwise distinct. If $m_{a,b}$ and $m_{a,c}$ are even, and $m_{b,c}$ is odd, then $m_{a,b} = m_{a,c}$.
\end{lemma}  

\begin{proof}
We will show that $m_{a,c}$ divides $m_{a,b}$. By symmetry, the same argument shows that $m_{a,b}$ divides $m_{a,c}$, hence $m_{a,b}=m_{a,c}$. Note that if $m_{a,c}=2$, then the conclusion is trivial, since $2$ divides every even number. Thus, we may assume $m_{a,c} > 2$.

Since $A_S$ is parabolic-retractable, there exists a retraction 
\[
\varphi \colon A_{\{a,b,c\}} \longrightarrow A_{\{a,c\}}
\]
such that $\varphi(a)=a$ and $\varphi(c)=c$.  

Restricting $\varphi$ to $A_{\{b,c\}}$, \autoref{theorem:classification_of_homo}(2) implies that $\varphi(b)=c$.  
Restricting $\varphi$ instead to $A_{\{a,b\}}$, \autoref{theorem:classification_of_homo}(6) yields two possibilities:

\begin{itemize}
\item[(a.1)] There exist $t_1,t_2 \in \mathbb{Z}$ such that 
\[
\varphi(b)=\delta_{a,c}^{t_1} a^{t_2},
\]
where $\delta_{a,c}$ denotes the generator of the center in $A_{\{a,c\}}$;
\item[(a.2)] There exist $\beta \in A_{\{a,c\}}$ and $t \in \mathbb{Z}$ such that 
\[
\varphi(b)=a^{-1}\beta (ac)^t \beta^{-1},
\]
and if $p=\gcd(t,\tfrac{m_{a,c}}{2})$ and $q=\tfrac{m_{a,c}}{2p}$, then $q \mid \tfrac{m_{a,b}}{2}$.
\end{itemize}

We claim that case (a.1) is impossible. Consider the homomorphism 
\[
\xi : A_{\{a,c\}} \longrightarrow \mathbb{Z}, \qquad \xi(a)=0,\ \xi(c)=1.
\]
This is well defined because $m_{a,c}$ is even. Since $\varphi(b)=c$, we should have $\xi(\varphi(b))=1$. However,
\[
\xi\big(\delta_{a,c}^{t_1} a^{t_2}\big) = t_1 \cdot \frac{m_{a,c}}{2},
\]
which is a multiple of $\tfrac{m_{a,c}}{2}>1$. Hence case (a.1) cannot occur.  

Thus we must be in case (a.2). In particular,
\[
c=\varphi(b)=a^{-1}\beta (ac)^t \beta^{-1}.
\]
Applying $\xi$ gives 
\[
1=\xi(c)=\xi(a^{-1}\beta (ac)^t \beta^{-1}) = t.
\]
Hence $p=\gcd(t,\tfrac{m_{a,c}}{2})=1$, so $q=\tfrac{m_{a,c}}{2}$. The condition that $q$ divides $\tfrac{m_{a,b}}{2}$ implies that $m_{a,c}$ divides $m_{a,b}$, as required.
\end{proof}

The next lemma gives one of the implications in \autoref{theoremB}.

\begin{lemma}\label{lemma_6}
Let $A_S$ be a parabolic-retractable Artin group. Then its associated Coxeter matrix $M$ is parabolic-retract-compatible.
\end{lemma}

\begin{proof}
Recall that $M$ is parabolic-retract-compatible if there exists a partition 
\[
S = T_1 \sqcup \dots \sqcup T_k
\]
such that:

\begin{enumerate}
    \item For each $i \in \{1,\dots,k\}$, the submatrix $M_{T_i}$ is retract-compatible; that is, all its entries are odd, and for any triple $a,b,c \in T_i$, one has that $m_{b,c}$ divides $m_{a,b}=m_{a,c}$, up to permutation.
    \item For each pair $i,j \in \{1,\dots,k\}$, with $i \neq j$, there exists $n_{ij}$, which is either an even number or $\infty$, such that $m_{a,b} = n_{ij}$ for all $(a,b)\in T_i \times T_j$.
\end{enumerate}

To construct the required partition, consider the Coxeter graph $\Gamma$ of $A_S$ and define a new graph $\Gamma_{\mathrm{odd}}$ with the same set of vertices as $\Gamma$, whose edges are exactly those labeled by odd integers. Let $\Gamma_1,\dots,\Gamma_k$ be the connected components of $\Gamma_{\mathrm{odd}}$, and define $T_i$ as the set of vertices of $\Gamma_i$.

By \autoref{lemma_1_odd_odd_odd}, each $\Gamma_i$ is a complete \emph{odd} Coxeter graph, i.e. all labels in $M_{T_i}$ are odd.
Furthermore, by \autoref{lemma_4_mbc_divides_odd_labels_all_cases}, each $M_{T_i}$ is retract-compatible. Thus condition (1) holds.

It remains to verify condition (2). Let $i \neq j$ and choose $a \in T_i$, $b \in T_j$. We already know that $m_{a,b}$ must be either even or $\infty$.  

If $m_{a,b}=\infty$, then for any $x \in T_i$, $y \in T_j$, we must also have $m_{x,y}=\infty$, as otherwise \autoref{lemma_2_odd_even_odd} would be contradicted.  

If $m_{a,b}$ is even, then by the same lemma all $m_{x,y}$ with $x \in T_i$, $y \in T_j$ must also be even. Moreover, \autoref{lemma_5_even_even_odd} implies that these even values must all coincide.

Thus, for each pair of distinct components $T_i, T_j$, all labels between them are equal to some $n_{ij} \in 2\mathbb{N} \cup \{\infty\}$. Hence condition (2) is satisfied, and the proof is complete.
\end{proof}

The next lemma is a technical remark that will be useful later.

\begin{lemma}\label{lemma sudoku}
Let $m,m'\in \N_{\ge 2}$. If $m$ divides $m'$ and $\Pi(a,b,m)=\Pi(b,a,m)$, then one also has $\Pi(a,b,m')=\Pi(b,a,m')$.
\end{lemma}

\begin{proof}
Write $m'=km$ for some $k\in \Z_{\ge 1}$.

\smallskip
\noindent\textbf{Case 1: $m$ even.}  
Then $m'$ is even and
\[
\Pi(a,b,m')=\bigl(\Pi(a,b,m)\bigr)^k=\bigl(\Pi(b,a,m)\bigr)^k=\Pi(b,a,m').
\]

\smallskip
\noindent\textbf{Case 2: $m$ odd and $m'$ even.}  
Then $k$ is even, and
\[
\Pi(a,b,m')=\bigl(\Pi(a,b,m)\,\Pi(b,a,m)\bigr)^{k/2}
=\bigl(\Pi(b,a,m)\,\Pi(a,b,m)\bigr)^{k/2}
=\Pi(b,a,m').
\]

\smallskip
\noindent\textbf{Case 3: $m$ odd and $m'$ odd.}  
Then $k$ is odd, and
\[
\begin{aligned}
\Pi(a,b,m')
&=\bigl(\Pi(a,b,m)\,\Pi(b,a,m)\bigr)^{\tfrac{k-1}{2}}\,\Pi(a,b,m) \\
&=\bigl(\Pi(b,a,m)\,\Pi(a,b,m)\bigr)^{\tfrac{k-1}{2}}\,\Pi(b,a,m)
=\Pi(b,a,m').
\end{aligned}
\]

In all cases we obtain $\Pi(a,b,m')=\Pi(b,a,m')$, as claimed.
\end{proof}

The following two lemmas explicitly construct retractions for Artin groups associated to retract-compatible Coxeter matrices.

\begin{lemma}\label{lemma_7_inductive_retract}
    Let $M$ be a retract-compatible Coxeter matrix, and let $A_S$ be its associated Artin group with $|S| \geq 2$. For every $x \in S$, there exists an (ordinary) retraction 
    \(\psi_x : A_S \rightarrow A_{S \setminus \{x\}}\) 
    such that \(\psi_x(x) \in S \setminus \{x\}\) is nontrivial.
\end{lemma}

\begin{proof}
First, we show that there exists \(y \in S \setminus \{x\}\) such that \(m_{x,y}\) divides \(m_{x,y'}\) for every \(y' \in S \setminus \{x\}\). Indeed, choose $y \in S \setminus \{x\}$ with $m_{x,y}$ minimal, that is, for all 
$y' \in S \setminus \{x\}$ we have $m_{x,y} \le m_{x,y'}$. 
Let $y' \in S \setminus \{x,y\}$. Then the hypothesis that $M$ is retract-compatible applied to $\{x,y,y'\}$ implies that either 
$m_{x,y} < m_{x,y'}$ and $m_{x,y} \mid m_{x,y'}$, or 
$m_{x,y} = m_{x,y'}$.


Next, we claim that for every \(z \in S \setminus \{x, y\}\), the value \(m_{y,z}\) divides \(m_{x,z}\). If \(m_{x,y} = m_{x,z}\), then retract-compatibility implies \(m_{y,z} \mid m_{x,z}\). Otherwise, since \(m_{x,y}\) divides \(m_{x,z}\) (by the choice of \(y\)) and the two values are distinct, retract-compatibility implies \(m_{x,z} = m_{y,z}\). This proves the claim.

We now define a map \(\psi_x\) by
\[
\psi_x(z) = z \quad \text{for all } z \in S \setminus \{x\}, \qquad \psi_x(x) = y.
\]
It remains to verify that \(\psi_x\) is a homomorphism. That is, we must check that
\[
\forall z \in S \setminus \{x\}, \quad \Pi(\psi_x(x), \psi_x(z), m_{x,z}) = \Pi(\psi_x(z), \psi_x(x), m_{x,z}).
\]
By the claim above, \(m_{y,z}\) divides \(m_{x,z}\), and by \autoref{lemma sudoku},
\[
\Pi(\psi_x(x), \psi_x(z), m_{x,z}) = \Pi(y, z, m_{x,z}) = \Pi(z, y, m_{x,z}) = \Pi(\psi_x(z), \psi_x(x), m_{x,z}),
\]
which completes the proof.
\end{proof}

\begin{lemma}\label{lemma_8_explicit_retract}
Let $M$ be a parabolic-retract-compatible Coxeter matrix, and let $A_S$ be its associated Artin group. Then, for every $x \in S$, there exists an ordinary retraction 
\[
\varphi_x: A_S \to A_{S \setminus \{x\}}.
\]
\end{lemma}

\begin{proof}
We have a partition $S = T_1 \sqcup \dots \sqcup T_k$ such that:

\begin{enumerate}
    \item For all $i \in \{1, \dots, k\}$, $M_{T_i}$ is retract-compatible.
    \item For each pair $i,j \in \{1, \dots, k\}$ with $i \neq j$, there exists $n_{ij}$, which is either an even number or $\infty$, such that $m_{a,b} = n_{ij}$ for all $(a,b) \in T_i \times T_j$. 
\end{enumerate}

Assume without loss of generality that $x \in T_1$. We consider two cases:

\begin{itemize}
    \item \textbf{Case 1:} $T_1 = \{x\}$. Define
    \[
    \varphi_x(z) = z \quad \text{for } z \in S \setminus \{x\}, \qquad \varphi_x(x) = 1.
    \]
    This is well-defined because, by parabolic-retract-compatibility, $x$ only shares even or $\infty$ relations with other generators. Hence, if $m_{x,z} \neq \infty$,
    \[
    \Pi(\varphi_x(x), \varphi_x(z), m_{x,z}) = z^{\frac{m_{x,z}}{2}} = \Pi(\varphi_x(z), \varphi_x(x), m_{x,z}).
    \]

    \item \textbf{Case 2:} $|T_1| \ge 2$. By \autoref{lemma_7_inductive_retract}, there exists an ordinary retraction 
    \[
    \psi_x: A_{T_1} \to A_{T_1 \setminus \{x\}}
    \] 
    such that $\psi_x(x) = y \in T_1 \setminus \{x\}$. Define
    \[
    \varphi_x(z) = 
    \begin{cases} 
    \psi_x(z) & \text{if } z \in T_1, \\
    z & \text{otherwise}.
    \end{cases}
    \]
    To check that $\varphi_x$ is a homomorphism, consider $z \in S \setminus \{x\}$ with $m_{x,z} \neq \infty$:

    \begin{itemize}
        \item If $m_{x,z}$ is odd, then $z \in T_1$, and since $\psi_x$ is a homomorphism, one has
\begin{align*}
    \Pi(\varphi_x(x), \varphi_x(z), m_{x,z}) 
       & = \Pi(\psi_x(x), \psi_x(z), m_{x,z}) \\
       & = \Pi(\psi_x(z), \psi_x(x), m_{x,z}) \\
       & = \Pi(\varphi_x(z), \varphi_x(x), m_{x,z}).
\end{align*}
        
        \item If $m_{x,z}$ is even, then $z \in T_j$ for some $j \neq 1$, and $m_{x,z} = n_{1,j} = m_{y,z}$. Therefore,
        \[
        \Pi(\varphi_x(x), \varphi_x(z), m_{x,z}) 
        = \Pi(y, z, m_{y,z}) 
        = \Pi(z, y, m_{y,z}) 
        = \Pi(\varphi_x(z), \varphi_x(x), m_{x,z}).
        \]
    \end{itemize}
\end{itemize}
\end{proof}

Now we are ready to prove our main theorems.

\begin{proof}[Proof of \autoref{theoremB}]
One direction follows from \autoref{lemma_6}. 

Now, suppose that $M$ is parabolic-retract-compatible. Notice that any submatrix of $M$ corresponding to a standard parabolic subgroup of $A_S$ is again parabolic-retract-compatible. By \autoref{lemma_8_explicit_retract}, for any $X \subseteq S$ and any $x \in X$, there exists a retraction 
\[
A_X \to A_{X \setminus \{x\}}.
\] 
Then, applying \autoref{lemma_0_induction}, we conclude that $A_S$ is parabolic-retractable.
\end{proof}

\begin{proof}[Proof of \autoref{theoremA}]
If $A_S$ is parabolic-retractable, then, by \autoref{lemma_0_induction}, each subgroup $A_{a,b,c}$ retracts onto $A_{b,c}$.

Next, we show that the fact that every $A_{a,b,c}$ is parabolic-retractable, which is equivalent to the condition stated in the theorem, implies that the Coxeter matrix $M$ associated to $A_S$ is parabolic-retract-compatible. Once this is established, \autoref{theoremB} implies that $A_S$ is parabolic-retractable.

The argument is essentially the same as the one in the proof of \autoref{lemma_6}. 
Consider the Coxeter graph $\Gamma$ of $A_S$ and define a new graph $\Gamma_{\mathrm{odd}}$ with the same set of vertices as $\Gamma$, whose edges are exactly those labeled by odd integers. 
Let $\Gamma_1, \ldots, \Gamma_k$ be the connected components of $\Gamma_{\mathrm{odd}}$, and define $T_i$ as the set of vertices of $\Gamma_i$. 
We have a partition $S = T_1 \sqcup \cdots \sqcup T_k$, and, again by \autoref{lemma_1_odd_odd_odd}, each $\Gamma_i$ is a complete Coxeter graph, i.e. all labels in $M_{T_i}$ are odd.


If $m_{a,b}, m_{a,c}$, and $m_{b,c}$ are all odd, then by \autoref{lemma_4_mbc_divides_odd_labels_all_cases}, up to a permutation of $a, b, c$, we have
\[
m_{a,b} = m_{a,c} \quad \text{and} \quad m_{b,c} \mid m_{a,b}.
\] 
This guarantees that each $M_{T_i}$ is retract-compatible, which satisfies the first condition of parabolic-retract-compatibility.

For the second condition, \autoref{lemma_2_odd_even_odd} shows that either there is no edge connecting vertices of $T_i$ and $T_j$, or all vertices of $T_i$ are connected to all vertices of $T_j$ by edges with even labels. Then, by \autoref{lemma_5_even_even_odd}, all these even-labeled edges must have the same label, which satisfies the second condition for parabolic-retract-compatibility.
\end{proof}

\begin{proof}[Proof of \autoref{theoremC}]
Suppose $A_S$ is a parabolic-retractable Artin group. By \autoref{lemma_6}, the associated Coxeter matrix $M$ is parabolic-retract-compatible. We will construct an ordinary retraction for every subset $X \subseteq S$. 

Write $X = S \setminus \{x_1, \dots, x_p\}$. Since any submatrix of $M$ corresponding to a standard parabolic subgroup is again parabolic-retract-compatible, \autoref{lemma_8_explicit_retract} guarantees that, for each $k = 1, \ldots, p$, there is an ordinary retraction 
\[
\varphi_k : A_{S \setminus \{x_1, \ldots, x_{k-1}\}} \to A_{S \setminus \{x_1, \ldots, x_k\}}.
\] 
Composing these retractions, the desired retraction is 
$\varphi_p \circ \cdots \circ \varphi_1$.
\end{proof}

\bibliography{main}

\end{document}

%% file: example1.pdf_tex
\begingroup%
  \makeatletter%
  \providecommand\color[2][]{%
    \errmessage{(Inkscape) Color is used for the text in Inkscape, but the package 'color.sty' is not loaded}%
    \renewcommand\color[2][]{}%
  }%
  \providecommand\transparent[1]{%
    \errmessage{(Inkscape) Transparency is used (non-zero) for the text in Inkscape, but the package 'transparent.sty' is not loaded}%
    \renewcommand\transparent[1]{}%
  }%
  \providecommand\rotatebox[2]{#2}%
  \newcommand*\fsize{\dimexpr\f@size pt\relax}%
  \newcommand*\lineheight[1]{\fontsize{\fsize}{#1\fsize}\selectfont}%
  \ifx\svgwidth\undefined%
    \setlength{\unitlength}{431.9423261bp}%
    \ifx\svgscale\undefined%
      \relax%
    \else%
      \setlength{\unitlength}{\unitlength * \real{\svgscale}}%
    \fi%
  \else%
    \setlength{\unitlength}{\svgwidth}%
  \fi%
  \global\let\svgwidth\undefined%
  \global\let\svgscale\undefined%
  \makeatother%
  \begin{picture}(1,0.35481692)%
    \lineheight{1}%
    \setlength\tabcolsep{0pt}%
    \put(0,0){\includegraphics[width=\unitlength,page=1]{example1.pdf}}%
    \put(0.84997731,0.17208902){\color[rgb]{0,0,0}\makebox(0,0)[lt]{\lineheight{1.25}\smash{\begin{tabular}[t]{l}$5$\end{tabular}}}}%
    \put(0.59641432,0.17208902){\color[rgb]{0,0,0}\makebox(0,0)[lt]{\lineheight{1.25}\smash{\begin{tabular}[t]{l}$25$\end{tabular}}}}%
    \put(0.72548311,0.07665723){\color[rgb]{0,0,0}\makebox(0,0)[lt]{\lineheight{1.25}\smash{\begin{tabular}[t]{l}$25$\end{tabular}}}}%
    \put(0.34989996,0.17208903){\color[rgb]{0,0,0}\makebox(0,0)[lt]{\lineheight{1.25}\smash{\begin{tabular}[t]{l}$75$\end{tabular}}}}%
    \put(0.62109646,0.01915655){\color[rgb]{0,0,0}\makebox(0,0)[lt]{\lineheight{1.25}\smash{\begin{tabular}[t]{l}$75$\end{tabular}}}}%
    \put(0.50354193,0.07427427){\color[rgb]{0,0,0}\makebox(0,0)[lt]{\lineheight{1.25}\smash{\begin{tabular}[t]{l}$75$\end{tabular}}}}%
    \put(0.09613089,0.17208903){\color[rgb]{0,0,0}\makebox(0,0)[lt]{\lineheight{1.25}\smash{\begin{tabular}[t]{l}$375$\end{tabular}}}}%
    \put(0.21458259,0.20692778){\color[rgb]{0,0,0}\makebox(0,0)[lt]{\lineheight{1.25}\smash{\begin{tabular}[t]{l}$375$\end{tabular}}}}%
    \put(0.33535691,0.25570201){\color[rgb]{0,0,0}\makebox(0,0)[lt]{\lineheight{1.25}\smash{\begin{tabular}[t]{l}$375$\end{tabular}}}}%
    \put(0.41500661,0.31965128){\color[rgb]{0,0,0}\makebox(0,0)[lt]{\lineheight{1.25}\smash{\begin{tabular}[t]{l}$375$\end{tabular}}}}%
  \end{picture}%
\endgroup%

%% file: example2.pdf_tex
\begingroup%
  \makeatletter%
  \providecommand\color[2][]{%
    \errmessage{(Inkscape) Color is used for the text in Inkscape, but the package 'color.sty' is not loaded}%
    \renewcommand\color[2][]{}%
  }%
  \providecommand\transparent[1]{%
    \errmessage{(Inkscape) Transparency is used (non-zero) for the text in Inkscape, but the package 'transparent.sty' is not loaded}%
    \renewcommand\transparent[1]{}%
  }%
  \providecommand\rotatebox[2]{#2}%
  \newcommand*\fsize{\dimexpr\f@size pt\relax}%
  \newcommand*\lineheight[1]{\fontsize{\fsize}{#1\fsize}\selectfont}%
  \ifx\svgwidth\undefined%
    \setlength{\unitlength}{457.30227493bp}%
    \ifx\svgscale\undefined%
      \relax%
    \else%
      \setlength{\unitlength}{\unitlength * \real{\svgscale}}%
    \fi%
  \else%
    \setlength{\unitlength}{\svgwidth}%
  \fi%
  \global\let\svgwidth\undefined%
  \global\let\svgscale\undefined%
  \makeatother%
  \begin{picture}(1,0.71493671)%
    \lineheight{1}%
    \setlength\tabcolsep{0pt}%
    \put(0,0){\includegraphics[width=\unitlength,page=1]{example2.pdf}}%
    \put(0.15837456,0.11252882){\color[rgb]{0,0,0}\makebox(0,0)[lt]{\lineheight{1.25}\smash{\begin{tabular}[t]{l}$5$\end{tabular}}}}%
    \put(0.15877025,0.34345132){\color[rgb]{0,0,0}\makebox(0,0)[lt]{\lineheight{1.25}\smash{\begin{tabular}[t]{l}$25$\end{tabular}}}}%
    \put(0.05965584,0.22282219){\color[rgb]{0,0,0}\makebox(0,0)[lt]{\lineheight{1.25}\smash{\begin{tabular}[t]{l}$25$\end{tabular}}}}%
    \put(0.15877026,0.57244871){\color[rgb]{0,0,0}\makebox(0,0)[lt]{\lineheight{1.25}\smash{\begin{tabular}[t]{l}$75$\end{tabular}}}}%
    \put(0.01047237,0.33167699){\color[rgb]{0,0,0}\makebox(0,0)[lt]{\lineheight{1.25}\smash{\begin{tabular}[t]{l}$75$\end{tabular}}}}%
    \put(0.05612292,0.43245559){\color[rgb]{0,0,0}\makebox(0,0)[lt]{\lineheight{1.25}\smash{\begin{tabular}[t]{l}$75$\end{tabular}}}}%
    \put(0,0){\includegraphics[width=\unitlength,page=2]{example2.pdf}}%
    \put(0.53722021,0.33946406){\color[rgb]{0,0,0}\makebox(0,0)[lt]{\lineheight{1.25}\smash{\begin{tabular}[t]{l}$7$\end{tabular}}}}%
    \put(0,0){\includegraphics[width=\unitlength,page=3]{example2.pdf}}%
    \put(0.84627395,0.44434801){\color[rgb]{0,0,0}\makebox(0,0)[lt]{\lineheight{1.25}\smash{\begin{tabular}[t]{l}$21$\end{tabular}}}}%
    \put(0.94262382,0.33534079){\color[rgb]{0,0,0}\makebox(0,0)[lt]{\lineheight{1.25}\smash{\begin{tabular}[t]{l}$21$\end{tabular}}}}%
    \put(0.85028476,0.22789584){\color[rgb]{0,0,0}\makebox(0,0)[lt]{\lineheight{1.25}\smash{\begin{tabular}[t]{l}$3$\end{tabular}}}}%
    \put(0,0){\includegraphics[width=\unitlength,page=4]{example2.pdf}}%
    \put(0.33171651,0.0962003){\color[rgb]{0,0,0}\makebox(0,0)[lt]{\lineheight{1.25}\smash{\begin{tabular}[t]{l}$4$\end{tabular}}}}%
    \put(0.25288625,0.24148522){\color[rgb]{0,0,0}\makebox(0,0)[lt]{\lineheight{1.25}\smash{\begin{tabular}[t]{l}$4$\end{tabular}}}}%
    \put(0.25984239,0.31605327){\color[rgb]{0,0,0}\makebox(0,0)[lt]{\lineheight{1.25}\smash{\begin{tabular}[t]{l}$4$\end{tabular}}}}%
    \put(0.28281077,0.39092137){\color[rgb]{0,0,0}\makebox(0,0)[lt]{\lineheight{1.25}\smash{\begin{tabular}[t]{l}$4$\end{tabular}}}}%
    \put(0.24799259,0.47473642){\color[rgb]{0,0,0}\makebox(0,0)[lt]{\lineheight{1.25}\smash{\begin{tabular}[t]{l}$4$\end{tabular}}}}%
    \put(0.30304748,0.51681025){\color[rgb]{0,0,0}\makebox(0,0)[lt]{\lineheight{1.25}\smash{\begin{tabular}[t]{l}$4$\end{tabular}}}}%
    \put(0.34695637,0.58620135){\color[rgb]{0,0,0}\makebox(0,0)[lt]{\lineheight{1.25}\smash{\begin{tabular}[t]{l}$4$\end{tabular}}}}%
    \put(0.29072486,0.16205247){\color[rgb]{0,0,0}\makebox(0,0)[lt]{\lineheight{1.25}\smash{\begin{tabular}[t]{l}$4$\end{tabular}}}}%
    \put(0.74306876,0.14405502){\color[rgb]{0,0,0}\makebox(0,0)[lt]{\lineheight{1.25}\smash{\begin{tabular}[t]{l}$6$\end{tabular}}}}%
    \put(0.74700012,0.21053021){\color[rgb]{0,0,0}\makebox(0,0)[lt]{\lineheight{1.25}\smash{\begin{tabular}[t]{l}$6$\end{tabular}}}}%
    \put(0.75593496,0.31995895){\color[rgb]{0,0,0}\makebox(0,0)[lt]{\lineheight{1.25}\smash{\begin{tabular}[t]{l}$6$\end{tabular}}}}%
    \put(0.77924667,0.39082674){\color[rgb]{0,0,0}\makebox(0,0)[lt]{\lineheight{1.25}\smash{\begin{tabular}[t]{l}$6$\end{tabular}}}}%
    \put(0.7482665,0.48096773){\color[rgb]{0,0,0}\makebox(0,0)[lt]{\lineheight{1.25}\smash{\begin{tabular}[t]{l}$6$\end{tabular}}}}%
    \put(0.76183801,0.56907091){\color[rgb]{0,0,0}\makebox(0,0)[lt]{\lineheight{1.25}\smash{\begin{tabular}[t]{l}$6$\end{tabular}}}}%
  \end{picture}%
\endgroup%

%% file: example3.pdf_tex
\begingroup%
  \makeatletter%
  \providecommand\color[2][]{%
    \errmessage{(Inkscape) Color is used for the text in Inkscape, but the package 'color.sty' is not loaded}%
    \renewcommand\color[2][]{}%
  }%
  \providecommand\transparent[1]{%
    \errmessage{(Inkscape) Transparency is used (non-zero) for the text in Inkscape, but the package 'transparent.sty' is not loaded}%
    \renewcommand\transparent[1]{}%
  }%
  \providecommand\rotatebox[2]{#2}%
  \newcommand*\fsize{\dimexpr\f@size pt\relax}%
  \newcommand*\lineheight[1]{\fontsize{\fsize}{#1\fsize}\selectfont}%
  \ifx\svgwidth\undefined%
    \setlength{\unitlength}{165.18829946bp}%
    \ifx\svgscale\undefined%
      \relax%
    \else%
      \setlength{\unitlength}{\unitlength * \real{\svgscale}}%
    \fi%
  \else%
    \setlength{\unitlength}{\svgwidth}%
  \fi%
  \global\let\svgwidth\undefined%
  \global\let\svgscale\undefined%
  \makeatother%
  \begin{picture}(1,0.78356278)%
    \lineheight{1}%
    \setlength\tabcolsep{0pt}%
    \put(0,0){\includegraphics[width=\unitlength,page=1]{example3.pdf}}%
    \put(0.47885168,0.06097409){\color[rgb]{0,0,0}\makebox(0,0)[lt]{\lineheight{1.25}\smash{\begin{tabular}[t]{l}$45$\end{tabular}}}}%
    \put(0,0){\includegraphics[width=\unitlength,page=2]{example3.pdf}}%
    \put(0.76760571,0.39513077){\color[rgb]{0,0,0}\makebox(0,0)[lt]{\lineheight{1.25}\smash{\begin{tabular}[t]{l}$45$\end{tabular}}}}%
    \put(0.50552187,0.44754751){\color[rgb]{0,0,0}\makebox(0,0)[lt]{\lineheight{1.25}\smash{\begin{tabular}[t]{l}$45$\end{tabular}}}}%
    \put(0.26711813,0.2509847){\color[rgb]{0,0,0}\makebox(0,0)[lt]{\lineheight{1.25}\smash{\begin{tabular}[t]{l}$45$\end{tabular}}}}%
    \put(0.19355004,0.39007381){\color[rgb]{0,0,0}\makebox(0,0)[lt]{\lineheight{1.25}\smash{\begin{tabular}[t]{l}$5$\end{tabular}}}}%
    \put(0.68972809,0.21822418){\color[rgb]{0,0,0}\makebox(0,0)[lt]{\lineheight{1.25}\smash{\begin{tabular}[t]{l}$3$\end{tabular}}}}%
  \end{picture}%
\endgroup%